\documentclass{svmult}

\usepackage{amsmath}
\usepackage{amssymb}

\smartqed

\begin{document}

\title*{Necessary Optimality Condition for
a Discrete Dead Oil Isotherm Optimal Control Problem}

\titlerunning{Necessary Optimality Condition for
a Discrete Dead Oil Isotherm Problem}

\author{Moulay Rchid Sidi Ammi\inst{1}\and
Delfim F. M. Torres\inst{2}}

\institute{Department of Mathematics, University of Aveiro,
3810-193 Aveiro, Portugal \texttt{sidiammi@ua.pt} \and
Department of Mathematics, University of Aveiro, 3810-193 Aveiro,
Portugal \texttt{delfim@ua.pt}}

\maketitle

%%%%%%%%%%%%%%%%%%%%%%%%%%%%%%%%%%

\begin{abstract}
We obtain necessary optimality conditions for a semi-discretized
optimal control problem for the classical system of nonlinear
partial differential equations modelling the water-oil (isothermal
dead-oil model).

\keywords{extraction of hydrocarbons; dead oil isotherm problem;
optimality conditions.}
\end{abstract}

%%%%%%%%%%%%%%%%%%%%%%

\section{Introduction}

We study an optimal control problem in the discrete case whose
control system is given by the following system of nonlinear partial differential equations,
\begin{equation}
\label{P}
\begin{cases}
\partial_t u - \Delta \varphi(u) = div\left(g(u) \nabla p\right)
& \text{ in } Q_T = \Omega \times (0,T) \, , \\
\partial_t p - div\left(d(u) \nabla p\right) = f
& \text{ in } Q_T = \Omega \times (0,T) \, , \\
\left.u\right|_{\partial \Omega} = 0 \, , \quad \left.u\right|_{t=0} = u_0 \, , \\
\left.p\right|_{\partial \Omega} = 0 \, , \quad \left.p\right|_{t=0}
= p_0 \, ,
\end{cases}
\end{equation}
which result from a well established model for oil engineering
within the framework of the mechanics of a continuous medium
\cite{mad}. The domain $\Omega$ is an open bounded set in
$\mathbb{R}^2$ with a sufficiently smooth boundary. Further
hypotheses on the data of the problem will be specified later.

At the time of the first run of a layer, the flow of the crude oil
towards the surface is due to the energy stored in the gases under
pressure in the natural hydraulic system. To mitigate the
consecutive decline of production and the decomposition of the
site, water injections are carried out, well before the normal
exhaustion of the layer. The water is injected through wells with
high pressure, by pumps specially drilled to this end. The pumps
allow the displacement of the crude oil towards the wells of
production. More precisely, the problem consists in seeking the
admissible control parameters which minimize a certain objective
functional. In our problem, the main goal is to distribute
properly the wells in order to have the best extraction of the
hydrocarbons. For this reason, we consider a cost functional
containing different parameters arising in the process. To
address the optimal control problem, we use the Lagrangian method
to derive an optimality system: from the cost function we introduce a Lagrangian; then, we calculate the G\^ateaux derivative of the
Lagrangian with respect to its variables. This technique was
used, in particular, by A.~Masserey\ et al. for electromagnetic
models of induction heating \cite{Bodart,Masserey}, and by
H.-C.~Lee and T.~Shilkin for the thermistor problem
\cite{LeeShilkin}.

We consider the following cost functional:
\begin{equation}
\label{eq:cf} J(u,p,f) = \frac{1}{2} \left\|u -
U\right\|_{2,Q_T}^2 + \frac{1}{2} \left\|p - P\right\|_{2,Q_T}^2 +
\frac{\beta_1}{2} \left\|f\right\|_{2 q_0,Q_T}^{2 q_0} +
\frac{\beta_2}{2} \left\|\partial_t f\right\|_{2,Q_T}^{2} \, .
\end{equation}
The control parameters are the reduced saturation of oil $u$, the
pressure $p$, and $f$. The coefficients $\beta_1 > 0$ and $\beta_2
\geq 0$ are two coefficients of penalization, and  $q_0 > 1$.  The
first two terms in \eqref{eq:cf} allow to minimize the difference
between the reduced saturation of oil $u$, the global pressure $p$
and the given data $U$ and $P$. The third and fourth terms are used
to improve the quality of exploitation of the crude oil. We take
$\beta_2 = 0$ just for the sake of simplicity. It is important to emphasize that our
choice of the cost function is not unique. One can always add
additional terms of penalization to take into account other
properties which one may wish to control. Recently, we proved in
\cite{ds1} results of existence, uniqueness, and regularity of the
optimal solutions to the problem of minimizing \eqref{eq:cf} subject
to \eqref{P}, using the theory of parabolic problems
\cite{lsu,Lions}. Here, our goal is to obtain necessary optimality
conditions which may be easily implemented on a computer. More
precisely, we address the problem of obtaining necessary optimality
conditions for the semi-discretized time problem. 

In order to be able to solve problem \eqref{P}-\eqref{eq:cf}
numerically, we use discretization of the problem in time by a
method of finite differences. For a fixed real $N$, let $\tau=
\frac{T}{N}$ be the step of a uniform partition of the interval
$[0, T]$ and $t_{n}=n\tau$, $n=1, \ldots ,N$. We denote by $u^{n}$
an approximation of $u$. The discrete cost functional is then
defined as follows:
\begin{equation}
\label{eq:cf1} J(u^{n},p^{n},f^{n}) = \frac{\tau}{2}\sum_{n=1}^{N}
\int_{\Omega} \big\{ \left\|u^{n} - U\right\|_{2,\Omega}^2 +
 \left\|p^{n} - P\right\|_{2,\Omega}^2 +\beta_1
\left\|f^{n}\right\|_{2 q_0,\Omega}^{2 q_0} \big\} \, dx \, .
\end{equation}
It is now possible to state our optimal control problem: find
$(\bar{u}^{n},\bar{p}^{n},\bar{f}^{n})$ which minimizes
\eqref{eq:cf1} among all functions $(u^{n},p^{n},f^{n})$
satisfying
\begin{equation}
\label{Pn}
\begin{cases}
\frac{u^{n+1}-u^{n}}{\tau} - \Delta \varphi(u^{n})
= div\left(g(u^{n}) \nabla p\right)
& \text{ in } \Omega  \, , \\
 \frac{p^{n+1}-p^{n}}{\tau}- div\left(d(u^{n}) \nabla p^{n}\right) = f^{n}
& \text{ in }  \Omega  \, , \\
\left.u\right|_{\partial \Omega} = 0 \, , \quad \left.u\right|_{t=0} = u_0 \, , \\
\left.p\right|_{\partial \Omega} = 0 \, , \quad
\left.p\right|_{t=0} = p_0 \, .
\end{cases}
\end{equation}
The soughtafter necessary optimality conditions are proved in
\S\ref{sec:MR} under suitable hypotheses on the data of the
problem.

%%%%%%%%%%%%%%%%%%%%%%%%%%%%%%%%%%%%%%%%%%%%

\section{Notation, hypotheses, and functional spaces}

Our main objective is to obtain necessary conditions for a triple
$\left(\bar{u}^{n},\bar{p}^{n},\bar{f}^{n}\right)$ to minimize
\eqref{eq:cf1} among all the functions
$\left(u^{n},p^{n},f^{n}\right)$ verifying \eqref{Pn}. In the sequel we assume that $\varphi$, $g$ and $d$ are real valued functions, respectively of class
$C^3$, $C^2$ and $C^1$, satisfying:

\begin{description}

\item[(H1)] $0 < c_1 \le d(r)$, $\varphi(r) \le c_2$;
$|d'(r)|,\, |\varphi'(r)|,\, |\varphi''(r)| \leq c_{3}\,  \quad
\forall r \in \mathbb{R}$.

\item[(H2)] $u_0$, $p_0$ $\in C^2\left(\bar{\Omega}\right)$, and
$U,\, P \in L^{2}(\Omega)$, where
  $u_0,\, p_0, \, U, \, P :
\Omega \rightarrow \mathbb{R}$,  and $\left.u_0\right|_{\partial
\Omega} = \left.p_0\right|_{\partial \Omega} = 0$.

\end{description}

We consider the following spaces:
\begin{equation*}
W_p^{1}(\Omega) := \left\{ u \in L^p(\Omega), \, \nabla u \in
L^p(\Omega) \right\} \, ,
\end{equation*}
endowed with the norm $\left\|u\right\|_{W_p^{1}(\Omega)} =
\left\|u\right\|_{p,\Omega} + \left\|\nabla u\right\|_{p,\Omega}$;
\begin{equation*}
W_p^{2}(\Omega) := \left\{ u \in W_p^{1}(\Omega), \, \nabla^2 u \,
\in L^p(\Omega) \right\} \, ,
\end{equation*}
with the norm $\left\|u\right\|_{W_p^{2}(\Omega)} =
\left\|u\right\|_{W_p^{1}(\Omega)} + \left\|\nabla^2
u\right\|_{p,\Omega} $; and the following notation:
\begin{gather*}
W :=  \stackrel{\circ}{W}_{2 q}^{2}(\Omega) \,
 ; \\
\Upsilon :=  L^{2 q}(\Omega)  \, ; \\
H := L^{2 q}(\Omega) \times \stackrel{\circ}{W}_{2 q}^{2 -
\frac{1}{q}}(\Omega) \, .
\end{gather*}

%%%%%%%%%%%%%%%%%%%%%

\section{Main results}
\label{sec:MR}

We define the following nonlinear operator corresponding to
\eqref{Pn}:
\begin{gather*}
F : W \times W \times \Upsilon \longrightarrow H \times H \\
\left(u^{n},p^{n},f^{n}\right) \longrightarrow
F(u^{n},p^{n},f^{n}) \, ,
\end{gather*}
where
\begin{equation*}
F\left(u^{n},p^{n},f^{n}\right) =
\left(%
\begin{array}{cc}
\frac{u^{n+1}-u^{n}}{\tau}  - \Delta \varphi(u^{n})
- div(g(u^{n}) \nabla p^{n}), & \gamma_0 u^{n} - u_0 \\
\frac{u^{n+1}-u^{n}}{\tau} - div\left(d(u^{n})
\nabla p^{n}\right) - f^{n}, & \gamma_0 p^{n} - p_0 \\
\end{array}%
\right) \, ,
\end{equation*}
$\gamma_0$ being the trace operator $\gamma_0 u^{n} =
\left.u\right|_{t=0}$. Our hypotheses ensure that $F$ is well
defined.

%%%%%%%%%%%%%%%%%%%%%%%%%%%%%%%%%%%%%%%%%%%%%%%%%%

\subsection{G\^{a}teaux differentiability}

\begin{theorem}
\label{thm5.1} In addition to the hypotheses (H1) and (H2), let us
suppose that
\begin{description}
\item[(H3)] $\left|\varphi'''\right| \le c$.
\end{description}
Then, the operator $F$ is G\^{a}teaux differentiable and for all
$(e,w,h) \in W \times W \times \Upsilon$ its derivative is given
by
\begin{equation*}
\begin{split}
\delta F(u^{n},p^{n},f^{n})(e,w,h) &= \frac{d}{ds} F\left(u^{n} +
s e, p^{n} + s w, f^{n} + s h\right)\left.\right|_{s=0} \\
&= \left(\delta F_1, \delta F_2\right)
= \left(%
\begin{array}{cc}
\xi_1 , & \xi_2  \\
\xi_3 , & \xi_4 \\
\end{array}%
\right) \, ,
\end{split}
\end{equation*}
$\xi_1 = e - div\left(\varphi'(u^{n})\nabla e\right) -
div\left(\varphi''(u^{n}) e \nabla u^{n}\right)$ $-
div\left(g(u^{n}) \nabla w\right)$ $- div\left(g'(u^{n}) e \nabla
p^{n}\right)$, $\xi_2 = \gamma_0 e$, $\xi_3 =  w -
div\left(d(u^{n}) \nabla w\right) - div\left(d'(u^{n}) e \nabla
p^{n}\right) - h$, $\xi_4 = \gamma_0 w$. Furthermore, for any
optimal solution
$\left(\bar{u}^{n},\bar{p}^{n},\bar{f}^{n}\right)$ of the problem
of minimizing \eqref{eq:cf1} among all the functions
$\left(u^{n},p^{n},f^{n}\right)$ satisfying \eqref{Pn}, the image
of $\delta F\left(\bar{u}^{n},\bar{p}^{n},\bar{f}^{n}\right)$ is
equal to $H \times H$.
\end{theorem}

To prove Theorem~\ref{thm5.1} we make use of the following lemma.

\begin{lemma}
\label{lemma5.2} The operator $\delta F(u^{n},p^{n},f^{n}) : W \times W \times
\Upsilon \longrightarrow H \times H$ is  linear and bounded.
\end{lemma}

\begin{proof}[Lemma~\ref{lemma5.2}]
For all $(e,w,h) \in W \times W \times \Upsilon $
\begin{multline*}
\delta_{u^{n}}F_{1}(u^{n}, p^{n}, f^{n})(e, w, h) \\
= e - div\left( \varphi'(u^{n}) \nabla e\right)- div\left(
\varphi''(u^{n})e \nabla u^{n} \right) \\
- div\left( g(u^{n}) \nabla
w\right)- div\left( g'(u^{n})e \nabla p^{n} \right) \\
= e- \varphi'(u^{n})\triangle e -\varphi''(u^{n}) \nabla u^{n}.
\nabla e -\varphi''(u^{n})e \triangle u^{n}\\ -\varphi''(u^{n})
\nabla e. \nabla u^{n} - \varphi'''(u^{n})e |\nabla u^{n}|^{2}
-g(u^{n})\triangle w- g'(u^{n}) \nabla u^{n}. \nabla w\\
- g'(u^{n}) e \triangle p^{n}- g'(u^{n})\nabla e. \nabla p^{n} -g''(u^{n})e \nabla u^{n}.
\nabla p^{n} \, ,
\end{multline*}
where $\delta_{u^{n}} F$ is the G\^ateaux derivative of $F$ with
respect to $u^{n}$. Using our hypotheses we have
\begin{equation*}
\begin{split}
\|g''(u^{n}) e \nabla u^{n}. \nabla p^{n}\|_{2q, \Omega} &\leq
\|e\|_{\infty, \Omega} \|
\nabla u^{n}. \nabla p^{n}\|_{2q, \Omega}\\
&\leq  \|e\|_{\infty, \Omega} \|\nabla u^{n}\|_{\frac{4q}{2-q},
\Omega} \|\nabla
p^{n}\|_{4, \Omega}\\
&\leq c \|u^{n}\|_{W} \|p^{n}\|_{W} \|e\|_{W}\, .
\end{split}
\end{equation*}
Evaluating each term of $\delta_{u^{n}}F_{1}$, we obtain
\begin{multline}
\label{eq:4.14} \|\delta_{u^{n}}F_{1}(u^{n}, p^{n}, f^{n})(e, w, h) \|_{2q, Q_{T}} \\
\leq c \left( \|u^{n}\|_{W}, \|p^{n}\|_{W},
\|f^{n}\|_{\Upsilon} \right) \left( \|e\|_{W}+ \|w\|_{W}+
\|h\|_{\Upsilon} \right)\, .
\end{multline}
In a similar way, we have for all $(e,w,h) \in W \times W \times
\Upsilon$ that
\begin{multline*}
\delta_{p^{n}}F_{2}(u^{n}, p^{n}, f^{n})(e, w, h)= w - div\left(
d(u^{n})
\nabla w\right)- div\left( d'(u^{n})e \nabla p^{n} \right)-h\\
= w- d(u^{n})\triangle w -d'(u^{n}) \nabla u^{n}. \nabla w  -
d'(u^{n})e \triangle p^{n} \\
-d'(u^{n}) \nabla e. \nabla u^{n}-
d'(u^{n})e \nabla u^{n}. \nabla p^{n} -h \, ,
\end{multline*}
with $\delta_{p^{n}} F$ the G\^ateaux derivative of $F$ with
respect to $p^{n}$. Then, using again our hypotheses, we obtain
that
\begin{multline}
\label{eq:4.12} \|\delta_{p^{n}}F_{2}(u^{n}, p^{n}, f^{n})(e, w, h) \|_{2q,
\Omega}\leq \|w\|_{2q, \Omega}+
\| \nabla w \|_{2q, \Omega}+ c \| \triangle w \|_{2q, \Omega}\\
+ c  \|\nabla u^{n}. \nabla w\|_{2q, \Omega}
+c \| e \triangle p^{n} \|_{2q, \Omega}\\
+ c  \|\nabla e. \nabla u^{n} \|_{2q, \Omega} + c  \| e \nabla
u^{n}. \nabla p^{n}\|_{2q, \Omega}+ \|h\|_{2q, \Omega} \, .
\end{multline}
Applying similar arguments to all terms of \eqref{eq:4.12}, we
then have
\begin{multline}
\label{eq:4.13} \|\delta_{p^{n}}F_{2}(u^{n}, p^{n}, f^{n})(e, w,
h) \|_{2q, \Omega} \\
\leq c \left( \|u^{n}\|_{W}, \|p^{n}\|_{W},
\|f^{n}\|_{\Upsilon} \right) \left( \|e\|_{W} +\|w\|_{W}+
\|h\|_{\Upsilon} \right)\, .
\end{multline}
Consequently, by \eqref{eq:4.14} and \eqref{eq:4.13} we can write
\begin{multline*}
\|\delta F(u^{n}, p^{n}, f^{n})(e, w, h) \|_{H\times H\times
\Upsilon} \\
\leq c \left( \|u^{n}\|_{W}, \|p^{n}\|_{W},
\|f^{n}\|_{\Upsilon} \right) \left( \|e\|_{W}+ \|w\|_{W}+
\|h\|_{\Upsilon} \right)\, .
\end{multline*}
\qed
\end{proof}

\begin{proof}[Theorem~\ref{thm5.1}]
In order to show that the image of $\delta F(\overline{u},
\overline{p}, \overline{f})$ is equal to $H \times H$, we need to
prove that there exists $(e, w, h) \in W \times W \times \Upsilon$
such that
\begin{equation}
\label{eq:4.15}
\begin{gathered}
\begin{split}
e - div\left( \varphi'(\overline{u^{n}}) \nabla e\right)
&- div\left( \varphi''(\overline{u^{n}})e \nabla \overline{u^{n}} \right) \\
&- div\left( g(\overline{u^{n}}) \nabla w\right)- div\left(
g'(\overline{u^{n}})e \nabla \overline{p^{n}} \right)= \alpha \, ,
\end{split} \\
w - div\left( d(\overline{u^{n}})
\nabla w\right)- div\left( d'(\overline{u^{n}})e
 \nabla \overline{p^{n}} \right)-h=
\beta  \, , \\
\left.e\right|_{\partial \Omega} = 0 \,
 , \quad \left.e\right|_{t=0} = b\, , \\
\left.w\right|_{\partial \Omega} = 0 \, ,  \quad
\left.w\right|_{t=0} = a \, ,
\end{gathered}
\end{equation}
for any $(\alpha, a)$ and $(\beta, b) \in H$. Writing the system
\eqref{eq:4.15} for $h=0$ as
\begin{equation}
\label{eq:4.16}
\begin{gathered}
e - \varphi'(\overline{u^{n}})\triangle e -2
\varphi''(\overline{u^{n}})\nabla \overline{u^{n}}. \nabla e-
\varphi''(\overline{u^{n}})e \triangle \overline{u^{n}}-
\varphi'''(\overline{u^{n}})e |\nabla \overline{u^{n}}|^{2} \, ,\\
\begin{split}
-g(\overline{u^{n}}) \triangle w &- g'(\overline{u^{n}}) \nabla
\overline{u^{n}}. \nabla w- g'(\overline{u^{n}}) e \triangle
\overline{p^{n}} \\
&-g'(\overline{u^{n}})\nabla \overline{p^{n}}. \nabla e -
g''(\overline{u^{n}}) e \nabla \overline{u^{n}}. \nabla
\overline{p^{n}} = \alpha \, ,
\end{split} \\
\begin{split}
w - d(\overline{u^{n}}) \triangle w &-d'(\overline{u^{n}}) \nabla
\overline{u^{n}}. \nabla w - d'(\overline{u^{n}}) e \triangle
\overline{p^{n}} \\
&- d'(\overline{u^{n}}) \nabla \overline{u^{n}}. \nabla
\overline{e} -d'(\overline{u^{n}}) e \nabla \overline{u^{n}}.
\nabla \overline{p^{n}} =\beta  \, ,
\end{split} \\
\left.e\right|_{\partial \Omega} = 0 \,
 , \quad \left.e\right|_{t=0} = b\, , \\
\left.w\right|_{\partial \Omega} = 0 \, ,  \quad
\left.w\right|_{t=0} = a \, ,
\end{gathered}
\end{equation}
it follows from the regularity of the optimal solution that
$\varphi''(\overline{u^{n}}) \triangle \overline{u^{n}}$,
$\varphi'''(\overline{u^{n}}) |\nabla \overline{u^{n}}|^{2}$,
$g'(\overline{u^{n}})  \triangle \overline{p^{n}}$,
$g''(\overline{u^{n}}) \nabla \overline{u^{n}}. \nabla
\overline{p^{n}}$, $d'(\overline{u^{n}}) \triangle
\overline{p^{n}}$, and $d'(\overline{u^{n}}) \nabla
\overline{u^{n}}. \nabla \overline{p^{n}}$ belong to
$L^{2q_{0}}(\Omega)$; $\varphi''(\overline{u^{n}})\nabla
\overline{u^{n}}$, $g'(\overline{u^{n}})\nabla \overline{u^{n}}$,
$g'(\overline{u^{n}})\nabla \overline{p^{n}}$, and
$d'(\overline{u^{n}})\nabla \overline{u^{n}}$ belong to
$L^{4q_{0}}(\Omega)$. This ensures, in view of the results of \cite{lsu,Lions},
existence of a unique solution of the system \eqref{eq:4.16}. Hence,
there exists a $(e, w, 0)$ verifying \eqref{eq:4.15}. We conclude
that the image of $\delta F$ is equal to $H \times H$. \qed
\end{proof}

%%%%%%%%%%%%%%%%%%%%%%%%%%%%%%%%%%%%%%%%%%%%%%%%%%

\subsection{Necessary optimality condition}

We consider the cost functional $J: W\times  W \times \Upsilon
\rightarrow \mathbb{R}$ \eqref{eq:cf1} and the Lagrangian
$\mathcal{L}$ defined by
 $$
 \mathcal{L}\left(u^{n}, p^{n}, f^{n}, p_{1}, e_{1}, a, b\right)=
 J\left(u^{n}, p^{n}, f^{n} \right)+ \left\langle F(u^{n}, p^{n}, f^{n}),
 \left(\begin{array}{cc}  p_{1} & a \\
                  e_{1}, & b
\end{array}\right) \right\rangle\, ,
$$
where the bracket $\langle \cdot, \cdot \rangle$ denotes the
duality between $H$ and $H'$.
\begin{theorem}
\label{thm5.2}
Under hypotheses (H1)--(H3), if $\left(\overline{u^{n}},
\overline{p^{n}}, \overline{f^{n}}\right)$ is an optimal solution to
the problem of minimizing \eqref{eq:cf1} subject to \eqref{Pn}, then
there exist functions $(\overline{e_{1}}, \overline{p_{1}}) \in
W_{2}^{2}(\Omega) \times W_{2}^{2}(\Omega)$ satisfying the following
conditions:
\begin{equation}
\label{eq:4.19}
\begin{gathered}
\begin{split}
\overline{e_{1}} + div\left( \varphi'(\overline{u^{n}}) \nabla
e_{1}\right) &-d'(\overline{u^{n}}) \nabla \overline{p^{n}}.
\nabla \overline{p_{1}} -\varphi''(\overline{u^{n}}) \nabla
\overline{u^{n}}. \nabla
\overline{e_{1}} \\
&- g'(\overline{u^{n}})\nabla \overline{p^{n}}. \nabla
\overline{e_{1}}= \tau \sum_{n=1}^{N}(\overline{u^{n}}-U)\, ,
\end{split} \\
\left.\overline{e_{1}}\right|_{\partial \Omega} = 0 \,
 , \quad \left.\overline{e_{1}}\right|_{t=T} = 0\, ,\\
\overline{p_{1}} + div\left( d(\overline{u^{n}}) \nabla
\overline{p_{1}}\right)+ div\left( g(\overline{u^{n}}) \nabla
\overline{e_{1}}\right)=\tau \sum_{n=1}^{N} (\overline{p^{n}}-P) \, , \\
\left.\overline{p_{1}}\right|_{\partial \Omega} = 0 \, , \quad
\left.\overline{p_{1}}\right|_{t=T} = 0 \, , \\
 q_{0} \beta_{1}\tau \sum_{n=1}^{N}|\overline{f^{n}}|^{2q_{0}-2}\overline{f^{n}}
  = \overline{p_{1}} \, .
\end{gathered}
\end{equation}
\end{theorem}

\begin{proof}
Let $ \left( \overline{u^{n}}, \overline{p^{n}}, \overline{f^{n}}
\right)$ be an optimal solution to the problem of minimizing
\eqref{eq:cf1} subject to \eqref{Pn}. It is well known (\textrm{cf.
e.g.} \cite{fur}) that there exist Lagrange multipliers $\left(
(\overline{p_{1}}, \overline{a}), (\overline{e_{1}}, \overline{b})
\right) \in H' \times H'$ verifying
$$ \delta_{(u^{n}, p^{n}, f^{n})}\mathcal{L}
\left(\overline{u^{n}}, \overline{p^{n}}, \overline{f^{n}},
\overline{p_{1}}, \overline{e_{1}}, \overline{a}, \overline{b}
\right)(e, w, h)= 0 \, \quad \forall (e, w, h)\in W \times W \times
\Upsilon, $$ with $\delta_{(u^{n}, p^{n}, f^{n})}\mathcal{L}$ the
G\^{a}teaux derivative of $\mathcal{L}$ with respect to $(u^{n},
p^{n}, f^{n})$. This  leads to the following system:
\begin{equation*}
\begin{gathered}
\tau \sum_{n=1}^{N}  \int_{\Omega} \left( (\overline{u^{n}}-U)e
+(\overline{p^{n}}-P)w+ q_{0}
\beta_{1}|\overline{f^{n}}|^{2q_{0}-2}\overline{f^{n}}h
 \right)\, dx \\
\begin{split}
-\int_{\Omega} \Biggl( \Bigl( e &- div\left(
\varphi'(\overline{u^{n}}) \nabla e \right)- div\left(
\varphi''(\overline{u^{n}})e \nabla \overline{u^{n}}\right)\\
& - div\left( g(\overline{u^{n}}) \nabla w \right) - div\left(
g'(\overline{u^{n}}) e \nabla \overline{p^{n}}\right)
\Bigr)\overline{e_{1}} \Biggr) \, dx
\end{split} \\
-\int_{\Omega} \left( w - div\left( d(\overline{u^{n}}) \nabla w
\right)- div\left( d'(\overline{u^{n}}) e\nabla
\overline{p^{n}}\right)-h
\right)\overline{p_{1}}  \, dx  \\
-\langle \gamma_{0} e, \overline{a}\rangle + -\langle \gamma_{0} w,
\overline{b}\rangle =0 \, \quad \forall (e, w, h)\in W \times W
\times \Upsilon.
\end{gathered}
\end{equation*}
The above system is equivalent to the following one:
\begin{equation}
\label{eq:4.20}
\begin{gathered}
\int_{\Omega} \Biggl(\tau \sum_{n=1}^{N}
(\overline{u^{n}}-U)e-div\left( d'(\overline{u^{n}}) e\nabla
\overline{p^{n}}\right)\overline{p_{1}}+e \, \overline{e_{1}} -
 div\left(
\varphi'(\overline{u^{n}}) \nabla e \right) \overline{e_{1}} \\
-div\left( \varphi''(\overline{u^{n}})e \nabla
\overline{u^{n}}\right) \overline{e_{1}} -div\left(
g'(\overline{u^{n}}) e \nabla
\overline{p^{n}}\right)\overline{e_{1}} \Biggr) \,
 dx   \\
 +\int_{\Omega} \left( \tau \sum_{n=1}^{N}(\overline{p^{n}}-P)w +w \,
\overline{p_{1}}- div\left( d(\overline{u^{n}}) \nabla w \right)
\overline{p_{1}} - div\left( g(\overline{u^{n}}) \nabla w
\right)\overline{e_{1}} \right) \, dx   \\
 +\int_{\Omega} \left(q_{0}
\beta_{1}\tau
\sum_{n=1}^{N}|\overline{f^{n}}|^{2q_{0}-2}\overline{f^{n}}h -
\overline{p_{1}}
h \right)\, dx   \\
 + \langle \gamma_{0} e, \overline{a}\rangle +
\langle \gamma_{0} w, \overline{b}\rangle =0 \, \quad \forall (e, w,
h)\in W \times W \times \Upsilon.
\end{gathered}
\end{equation}
In others words, we have
\begin{equation}
\label{eq:4.21}
\begin{gathered}
\begin{split}
\int_{\Omega} \Bigg( \tau \sum_{n=1}^{N}(\overline{u^{n}}-U) &+
d'(u)\nabla \overline{p^{n}}. \nabla \overline{p_{1}}- \overline{
e_{1}} - div\left( \varphi'(\overline{u^{n}}) \nabla
\overline{e_{1}} \right) \\
&+ \varphi''(\overline{u^{n}}) \nabla \overline{u^{n}}. \nabla
\overline{e_{1}}+g'(u^{n})\nabla \overline{p^{n}}. \nabla
\overline{e_{1}} \Biggr) e \, dx
\end{split} \\
+\int_{\Omega} \left( \tau \sum_{n=1}^{N}(\overline{p^{n}}-P)
+\overline{p_{1}} - div\left( d(\overline{u^{n}}) \nabla
\overline{p_{1}} \right) - div\left( g(\overline{u^{n}}) \nabla
\overline{e_{1}} \right) \right)w
\, dx   \\
+\int_{\Omega} \left(q_{0} \beta_{1}\tau
\sum_{n=1}^{N}|\overline{f^{n}}|^{2q_{0}-2}\overline{f^{n}}h -
\overline{p_{1}}
h \right)\, dx   \\
+ \langle \gamma_{0} e, \overline{a}\rangle + \langle \gamma_{0} w,
\overline{b}\rangle =0 \, \quad
 \forall (e, w, h)\in W \times W \times \Upsilon.
\end{gathered}
\end{equation}
Consider now the system
\begin{equation}
\label{eq:4.22}
\begin{gathered}
\begin{split}
e_{1} +div\left( \varphi'(\overline{u^{n}}) \nabla e_{1} \right)
&- d'(\overline{u^{n}})\nabla \overline{p^{n}}. \nabla p_{1}
-\varphi''(\overline{u^{n}}) \nabla \overline{u^{n}}. \nabla e_{1}\\
&-g'(\overline{u^{n}})\nabla \overline{p^{n}}. \nabla e_{1}= \tau
\sum_{n=1}^{N} (\overline{u^{n}}-U)  \, ,
\end{split} \\
p_{1}+div\left( d(\overline{u^{n}}) \nabla p_{1} \right)
+div\left( g(\overline{u^{n}}) \nabla e_{1} \right)= \tau\sum_{n=1}^{N}(\overline{p^{n}}-P), \\
 \left.e_{1}\right|_{\partial \Omega}
 =\left.p_{1}\right|_{\partial \Omega}  = 0 \, ,
\quad \left.e_{1}\right|_{t=T} = \left.p_{1}\right|_{t=T} = 0 \, ,
\end{gathered}
\end{equation}
with unknowns $(e_{1}, p_{1})$ which is uniquely solvable in
$W_{2}^{2}(\Omega) \times W_{2}^{2}(\Omega)$ by the theory of
elliptic equations \cite{lsu}. The problem of finding $(e, w)
\in W \times W$ satisfying
\begin{equation}
\label{eq:4.23}
\begin{gathered}
\begin{split}
e -div\left( \varphi'(\overline{u^{n}}) \nabla e \right) &-
div\left( \varphi''(\overline{u^{n}}) e \nabla \overline{u^{n}}
\right)- div\left( g(\overline{u^{n}}) \nabla w \right)\\
&- div\left( g'(\overline{u^{n}})e \nabla \overline{p^{n}}
\right)= sign(e_{1}- \overline{e_{1}})\, ,
\end{split} \\
w -div\left(d(\overline{u^{n}}) \nabla w \right)- div\left(
d'(\overline{u^{n}})e \nabla \overline{p^{n}} \right)=
sign(p_{1}- \overline{p_{1}})\, , \\
\gamma_{0}e= \gamma_{0}w= 0 \, ,
\end{gathered}
\end{equation}
is also uniquely solvable on $W_{2q}^{2}(\Omega)\times
W_{2q}^{2}(\Omega)$. Let us choose $h=0$ in \eqref{eq:4.21} and multiply \eqref{eq:4.22} by $(e, w)$. Then, integrating by parts and making the difference with \eqref{eq:4.21} we obtain:
\begin{equation}
\label{eq:4.24}
\begin{gathered}
\begin{split}
\int_{\Omega} \Bigl(e -div\left( \varphi'(\overline{u^{n}}) \nabla
e \right) &- div\left( \varphi''(\overline{u^{n}}) e \nabla
\overline{u^{n}} \right)- div\left( g(\overline{u^{n}}) \nabla w
\right) \\
&- div\left( g'(\overline{u^{n}})e \nabla \overline{p^{n}} \right)
\Bigr)(e_{1}- \overline{e_{1}})\, dx
\end{split} \\
+\int_{\Omega}\left( w -div\left(d(\overline{u^{n}}) \nabla w
\right)- div\left( d'(\overline{u^{n}})e \nabla \overline{p^{n}}
\right)
\right)(p_{1}- \overline{p_{1}})\, dx   \\
+\langle \gamma_{0}e,  \gamma_{0}\overline{e_{1} } - \overline{a}
\rangle
 + \langle \gamma_{0}w,  \gamma_{0}\overline{p_{1} } - \overline{b}
\rangle =0 \, \quad \forall (e, w) \in W \times W.
\end{gathered}
\end{equation}
Choosing $(e, w)$ in \eqref{eq:4.24} as the solution of the system
\eqref{eq:4.23}, we have
 \begin{gather*}
\int_{\Omega} sign(e_{1}- \overline{e_{1}})(e_{1}-
\overline{e_{1}})\, dx dt
 + \int_{\Omega} sign(p_{1}- \overline{p_{1}})(p_{1}-
\overline{p_{1}})\, dx  = 0 \, .
\end{gather*}
It follows that $e_{1}=\overline{ e_{1}}$ and  $p_{1}=\overline{
p_{1}}$. Coming back to \eqref{eq:4.24}, we obtain
$\gamma_{0}\overline{e_{1}}= \overline{a}$ and
$\gamma_{0}\overline{p_{1}}= \overline{b}$. On the other hand,
choosing $(e, w)=(0, 0)$ in \eqref{eq:4.21}, we get
$$
\int_{\Omega}\left( \beta_{1} \tau
\sum_{n=1}^{N}|\overline{f^{n}}|^{2q_{0}-2}\overline{f^{n}}
-\overline{p_{1}} \right) h \, dx= 0, \, \forall  h \in \Upsilon.
$$
 Then
\eqref{eq:4.19} follows, which concludes the proof of
Theorem~\ref{thm5.2}.\qed
\end{proof}

We claim that the results we
obtain here are useful for numerical implementations. This is
still under investigation and will be addressed in a forthcoming publication.

%%%%%%%%%%%%%%%%%%%%%

\section*{Acknowledgments}

The authors were supported by the \emph{Portuguese Foundation 
for Science and Technology} (FCT) through the \emph{Centre for Research on Optimization and Control} (CEOC) of the University of Aveiro, cofinanced by the European Community fund 
FEDER/POCI 2010. This work was developed under the post-doc 
project SFRH/BPD/20934/2004.

%%%%%%%%%%%%%%%%%%%%%

%%%%%%%%%%%%%%%%%%%%%

\end{document}